\documentclass{article}
\usepackage[utf8]{inputenc}
\usepackage{graphicx}
\usepackage{mathtools,amsmath,amssymb,amsfonts,amsthm}
\usepackage{mathalfa}
\usepackage{tikz}
\usetikzlibrary{calc}
\usepackage{color}
\usepackage{url}

\usepackage{enumitem}

\usepackage{subcaption}

\graphicspath{{./figures/}}

\usepackage[english]{babel}

\usepackage{tabularx}

\usepackage{xcolor}

\usepackage{verbatim}
\usepackage{xspace}

\usepackage{amsfonts}
\usepackage{mathtools}
\usepackage{amsmath}

\usepackage{amssymb}

\usepackage{setspace}

\usepackage{hhline}

\usepackage{fullpage}
\usepackage{ifthen}
\usepackage{tikz}
\usetikzlibrary{automata,positioning, arrows}

\theoremstyle{plain}
\newtheorem{thm}{Theorem}
\newtheorem{lemma}{Lemma}
\newtheorem{prop}{Proposition}

\theoremstyle{definition}

\newtheorem{defn}{Definition}

\theoremstyle{remark}
\newtheorem{myexample}{Example}
\newtheorem{rmk}{Remark}
\newtheorem{assumption}{Assumption}

\newenvironment{example}[1][0]
{ 
  \ifthenelse{\equal{#1}{0}}{
  \myexample
}
{ 
  \renewcommand\themyexample{#1}\myexample
  \addtocounter{myexample}{-1}
}
}
{\endmyexample}

\newcommand{\virgolette}[1]{``#1''}

\newcommand{\R}{\mathbb{R}}{}
\newcommand{\N}{\mathbb{N}}

\newcommand{\E}{\mathbb{E}}
\newcommand{\bP}{\mathbb{P}}

\newcommand{\cA}{\mathcal{A}}
\newcommand{\cB}{\mathcal{B}}

\newcommand{\cD}{\mathcal{D}}
\newcommand{\cH}{\mathcal{H}}

\newcommand{\cL}{\mathcal{L}}

\newcommand{\cG}{\mathcal{G}}
\newcommand{\M}{\langle M \rangle}

\newcommand{\wi}{\hat {\imath}}
\newcommand{\wj}{\hat {\jmath}}

\newcommand{\cF}{\mathcal{F}}

\newcommand{\cS}{\mathcal{S}}
\newcommand{\cQ}{\mathcal{Q}}

\newcommand{\cV}{\mathcal{V}}
\newcommand{\cW}{\mathcal{W}}

\newcommand{\wt}{\widetilde}

\DeclareMathOperator{\en}{\text{\emph{end}}}
\DeclareMathOperator{\st}{\text{\emph{st}}}
\DeclareMathOperator{\lab}{\text{\emph{lab}}}
\DeclareMathOperator{\Pt}{\emph{Path}}

\DeclarePairedDelimiterX{\inp}[2]{\langle}{\rangle}{#1, #2}

\newcommand{\mat}[1]{{\color{black}#1}}
\author{Matteo Della Rossa \and Raph\"ael M. Jungers 
\thanks{
RJ is a FNRS honorary Research Associate. This project has received funding from the European Research Council (ERC) under the European Union's Horizon 2020 research and innovation programme under grant agreement No 864017 - L2C. RJ is also supported by the Innoviris Foundation and the FNRS (Chist-Era Druid-net). 
}
\thanks{M. Della Rossa and R. Jungers are with ICTEAM,
        UCLouvain, Louvain-la-Neuve, Belgium).
         {\tt\small matteo.dellarossa@uclouvain.be}}%
}

\title{Almost sure Stability of Stochastic Switched Systems: Graph lifts-based
Approach}

\begin{document}

\maketitle

\begin{abstract}
In this paper, we develop tools to establish almost sure stability of stochastic switched systems \mat{whose switching signal is constrained by an automaton}. After having provided the necessary generalizations of existing results in the setting of stochastic graphs, we provide a characterization of almost sure stability in terms of multiple Lyapunov functions. We introduce the concept of \emph{lifts}, providing formal expansions of stochastic graphs, together with the guarantee of conserving the underlying probability framework. We show how these techniques, firstly introduced in the deterministic setting, provide hierarchical methods in order to  compute tight upper bounds for the almost sure decay rate. The theoretical developments are finally illustrated via a numerical example.
\end{abstract}
\section{Introduction}
Switched linear systems provide a fruitful mathematical model for a large class of physical systems and have been the center of intense research in the last decades, for an overview, see~\cite{liberzon}. In this framework, given $M$ linear subdynamics $A_1,\dots, A_M\in \R^{n\times n}$, we consider the system
\begin{equation}\label{eq:SwitchedSysIntro}
x(k+1)=A_{\sigma(k)}x(k),\;\;\;\;k\in \N,
\end{equation}
where $\sigma:\N\to \{1,\dots, M\}$, the \emph{switching signal}, selects, at each instant of time, which subsystems the solution will follow.
 In some situations, we only have partial information on the reasonable/suitable signals $\sigma:\N\to \{1,\dots, M\}$, and thus the jumps among the subsystems are modeled as a \emph{stochastic process}, providing, at each instant of time, the \emph{probability} of having followed a particular switching policy. In this case, the arising stochastic system takes the name of \emph{discrete-time (Markov) jump linear system} or \emph{stochastic switched system}, and the earliest works studying stability in this setting can be found in~\cite{Ber60,KatsKras60,FurKes60}. The case of stochastic switching modeled by an i.i.d.  (independent and identically distributed) \mat{sequence} or by a Markov chain process is studied in~\cite{FangLop1995,CosFra93,Pro11,BolCol04,DaiHua08}. 

Since we consider stochasticity in system~\eqref{eq:SwitchedSysIntro}, several \emph{stability notions} can be defined and tackled. The case of \emph{worst case} stability is studied in~\cite{PhiEss16,LeeDull06}, providing a graph-constrained adaptation of the concept of \emph{joint spectral radius}~(JSR), see~\cite{Jung09} for a monograph.  The case of \emph{second moment stability}, (i.e. studying the convergence of the mean of the squared norm of solutions), is tackled in~\cite{FangLop1995,CosFra93,COsFraMar05}. In these works, it is proven that the second moment stability can be ensured, without conservatism, by semidefinite optimization, via a quadratic Lyapunov functions approach. The case of $q$-moment stability is tackled in~\cite{FangLop1995,JunPro11,FengLopJi92}, providing again a characterization in terms of a ``spectral quantity'' (\emph{$q$-radius}) and/or in terms of Lyapunov functions. 

In some situations, worst case or $q$-moment stability notions are restrictive or not meaningful for the considered systems; for that reason, recent research focused in studying \emph{almost sure stability}: system~\eqref{eq:SwitchedSysIntro} is said to be almost surely stable if the solutions are converging to zero, for almost all switching sequences (with respect to the underlying probability measure). This problem is studied for example in~\cite{Fang97,BolCol04,DaiHua08,Pro11,ProJun13,ChiMaz21}, and a common tool in this analysis is provided by the \emph{(maximal) Lyapunov exponent}, which provides a characterization of almost sure stability: indeed, the Lyapunov exponent can be seen, roughly speaking, as a probabilistic counterpart of the (logarithm of the) JSR. Unfortunately, it is well known that the problem of computing the Lyapunov exponent is NP-hard or even undecidable, see~\cite{TsiBlo97}.

In this manuscript, we provide novel techniques to estimate the Lyapunov exponent, proposing new sufficient conditions for the almost sure stability of~\eqref{eq:SwitchedSysIntro}. Our ideas rely on a \mat{language-theoretic} interpretation of Markov chains, and are inspired by \mat{techniques introduced in~\cite{PhiEss16} for the deterministic case.} We consider formal expansions of Markov chains, called \emph{lifts}, which, while not modifying the behavior of the stochastic system, simplify the task of providing upper bounds to the Lyapunov exponent. This requires to introduce a flexible stochastic model, generalizing the Markov chains framework, and we thus consider the \emph{stochastic graphs} formalism, see~\cite{LinMar95}. Then, our approach makes use of the concept of \emph{multiple Lyapunov functions}, introduced in~\cite{Pro10,Pro11} in the i.i.d. case, and adapted here in this general context. We prove that, increasing the dimension (in terms of node/edges) of the lift of the Markov chain, \mat{the corresponding estimation of the Lyapunov exponent is asymptotically exact.} 

The manuscript is organized as follows: in Section~\ref{sec:Prelim} we recall the necessary preliminaries of probability theory, while in Section~\ref{Sec:AlmostSUre} we provide the extension of results involving almost sure stability, the probabilistic spectral radius and multiple Lyapunov functions. In Section~\ref{sec:LIFt} we define the main concepts of our work, the \emph{lifts} of stochastic graphs, and we provide our main results, which are then illustrated in~Section~\ref{sec:NumExample} with the help of a numerical example.
\section{Preliminaries}\label{sec:Prelim}
In this section we introduce the necessary notation from  probability theory and stochastic switched systems. 
\subsection{Shift Space and Ergodicity}
Given $M\in \N$, consider the finite set $\M:=\{1,\dots, M\}$, and the \emph{one-sided Bernoulli space} defined by
$
\Sigma_M:=\{\sigma=(\sigma_0,\sigma_1,\sigma_2,\dots)\;\vert \;\forall k\in \N, \;\sigma_k\in \M\}$.
We consider the \emph{left-shift operator} $\ell:\Sigma_M\to \Sigma_M$ defined by $\ell(\sigma)=(\sigma_1,\sigma_2,\dots)$. 
 The \emph{Borel $\sigma$-algebra} of $\Sigma_M$, denoted by $\cB(\Sigma_M)$, is generated by the \emph{cylinders} of the form $[i_{k-1},\dots, i_{0}]:=\{\sigma\in \Sigma_M\;\vert\; \sigma_0=i_0,\,\dots\,,\sigma_{k-1}=i_{k-1}\}$, where $k\in \N$ is the length of the cylinder. There is an identification between cylinders of length $k$ and elements of $\M^k$ and thus we also denotes with $\wi=(i_{k-1},\dots, i_0)\in \M^k$ a generic cylinder $[\wi]$ of length $k$. A measure $\mu$ on $(\Sigma_M,\cB(\Sigma_M))$ is said to be \emph{shift-invariant} if $\mu(\ell^{-1}(C))=\mu(C)$ for all $C\in\cB(\Sigma_M)$. A shift-invariant measure $\mu$ is \emph{ergodic} with respect to $\ell$ if 
for all $C\in\cB(\Sigma_M)$ such that $\ell^{-1}(C)=C$ we have $(\mu(C)=0$  or $\mu(C)=1)$.
For an overview of this topic see~\cite[Chapter~1]{Walters2000}.

\subsection{Stochastic Graphs and Stochastic Switched Systems}
In the following, given $M\in \N$, we introduce the structure used to define probability measures on $(\Sigma_M, \cB(\Sigma_M))$.
\begin{defn}\label{defn:StochAsticGraph}
 Given $M\in \N$, a  \emph{stochastic graph} $\cG=(S,E,p)$ on $\M$ is defined by
 \begin{enumerate}[leftmargin=*]
  \item A finite set $S$, the set of nodes;
  \item The set $E\subset S\times S\times \M$ of directed, labeled edges;
  \item A function $p:E\to (0,1]$, where, given $e\in E$, $p(e)$ is the \emph{probability}  associated with $e\in E$. 
\end{enumerate}
 We denote the generic edge by $e=(a,b,i)$; $i=:\lab(e)\in \M$ is the \emph{label} of $e$, $a=:\st(e)\in S$ and $b=:\en(e)\in S$ are the \emph{starting} and \emph{ending} nodes of $e$, respectively. The probability $p(e)$, when needed, is also denoted by  $p_{a,b,i}$. We require that
\[
\sum_{b\in S, i\in \M}p_{a,b,i}=1,\;\;\;\forall a\in S.
\] 
\end{defn}
When needed for notational simplicity, we set $p_{a,b,i}=0$ for every $(a,b,i)\notin E$. For every $K\in \N$, by $\Pt^K(\cG)$ we denote the set of paths in $\cG$ of length $K$. Given $\overline q\in \Pt^K(\cG)$, $\overline q=(e_{j_1},\cdots, e_{j_K})$, we define $p(\overline q):=p(e_{j_1})\cdots p(e_{j_K})$. We denote by $\text{\emph{st}}(\overline q)\in S$ the \emph{starting node} of $\overline  q\in \Pt^K(\cG)$, and we say that $\wi=(i_{K-1},\dots, i_0)\in \M^K$ is the \emph{label} of $\overline q=(e_{j_1},\cdots, e_{j_K})\in \Pt^K(\cG)$ and we write $\text{\emph{lab}}(\overline q)=\wi$ if $\text{\emph{lab}}(e_{j_1})=i_0$, $\dots$,  $\text{\emph{lab}}(e_{j_K})=i_{K-1}$.

We define $\Xi_S$ the set of probability distributions on $S$, and we can identify $\Xi_S=\{ \xi\in \R^{|S|}_+\;\vert\;\sum_{j=1}^{|S|}\xi_j=1\}$. Given any $\xi \in \Xi_S$ and any stochastic graph $\cG$  on  $\M$,  we can define a probability measure on $(\Sigma_M,\cB(\Sigma_M))$ as clarified in what follows; for every $k\in \N$, we first consider a probability measure on $S \times \M^k$, denoted by $\bP_{\cG,\xi}$ (without making $k$ explicit, for simplicity) defined recursively as follows:
\begin{equation}\label{eq:ProbabilityDefns}
\begin{aligned}
\begin{cases}
\bP_{\cG,\xi}(s,\emptyset):=\xi(s),\;\;\;\forall\;s\in S,\\
\bP_{\cG,\xi}(s,i):=\sum_{a\in S}\xi(a)p_{a,s,i},\;\;\;\forall \;(s,i)\in S\times \M,\\
\bP_{\cG,\xi}(s,\wi):=\sum_{a\in S}\bP_{\cG,\xi}(a,\wi^-)p_{a,s,\wi_f},\;\forall s\in S,\;\wi\in \M^k
\end{cases}
\end{aligned}
\end{equation}
where, given $\wi=(i_{k-1},\dots, i_0)\in \M^k$, $\wi^-:=(i_{k-2},\dots, i_0)\in \M^{k-1}$ denotes the \emph{predecessor} of $\wi$ and $\wi_f:=i_{k-1}$ is the \emph{final label} of $\wi$. Intuitively, $\bP_{\cG,\xi}(s,\wi)$ denotes the probability of ``being'' in the node $s\in S$ after having followed a path labeled by the (multi-)index $\wi\in \M^k$, given an initial probability measure $\xi$.
Finally, we introduce $\mu_{\cG,\xi}$ on   $(\Sigma_M,\cB(\Sigma_M))$ by defining it on the set of cylinders of $\Sigma_M$, i.e., $\forall \;\wi\in \M^k$,  $\forall k\in \N$ we set
\begin{equation}\label{eq:ProbabilityDefnsPaths}
\mu_{\cG,\xi}(\wi):=\sum_{s\in S}\bP_{\cG,\xi}(s,\wi).
\end{equation}
Another possible definition of~\eqref{eq:ProbabilityDefnsPaths}, is obtained setting 
\begin{equation}\label{eq:DefProbabilityCountingPaths}
\mu_{\cG,\xi}(\wi):=\sum_{s\in S}\xi(s)\sum_{\substack{\overline q\in \Pt^k(\cG)\\\text{\emph{st}}(\overline q)=s,\, \text{\emph{lab}}(\overline q)=\wi }}p(\overline q),
\end{equation}
 for all $\wi\in \M^k$ and for all $k\in \N$.
It is easy to see that~\eqref{eq:ProbabilityDefnsPaths} and~\eqref{eq:DefProbabilityCountingPaths} are equivalent and in the following we may use both, depending on the convenience.
With this definition, for any stochastic graph $\cG$ and any $\xi\in \Xi_S$, we have that $(\Sigma_M,\cB(\Sigma_M),\mu_{\cG,\xi})$ is a \emph{well-defined probability space}.
\begin{rmk}[Choice of the model]
For more details regarding the stochastic graph formalism, see~\cite[Definition~2.3.14]{LinMar95} and references therein.
The case of finite Markov chain is recovered by the setting in Definition~\ref{defn:StochAsticGraph}; given a stochastic matrix $P=(p_{ij})\in \R^{M\times M}_{\geq 0}$, the state transition matrix associated to a time homogeneous Markov chain, we can define the corresponding stochastic graph by
$S:=\M$, $E=\{(i,j,j)\;\vert\;(i,j)\in \M^2\}$ and $p_{i,j,j}=p_{i,j}$, for all $(i,j)\in \M^2$.
For a graphical representation, see Figure~\ref{Fig:GraphH}. It can be seen that \emph{every} stochastic graph can be rewritten, paying the price of ``enlarging'' the alphabet and/or the node set, as an ordinary Markov chain, see~\cite{LinMar95}. We carry out the analysis in the setting of general stochastic graphs, since it will be crucial, in the following sections, when considering ``expanded'' version of a given graph/Markov chain. 

\end{rmk}
In the next developments, given a stochastic graph $\cG$, we consider the following crucial property.
\begin{assumption}\label{assum:StrongConn}
The considered stochastic graph $\cG$ is strongly connected, i.e., for all $a,b\in S$ there exists a path in $\cG$ starting at $a$ and arriving at $b$.
\end{assumption}
Given $\cG=(S,E,p)$ a stochastic graph on $\M$, let us consider the stochastic matrix $P_\cG\in \R^{|S|\times |S|}_{\geq 0}$ defined by
\begin{equation}\label{eq:DefnMatrixAssociatedtoG}
p_{a,b}:=\sum_{i\in \M}p_{a,b,i}.
\end{equation}
If $\cG$ satisfies Assumption~\ref{assum:StrongConn}, by the Perron-Frobenius Theorem, we can consider $\xi_\cG\in\Xi_S$ as the \emph{unique} measure such that $\xi_\cG^\top P_\cG=\xi_\cG^\top$, which satisfies ${\xi_\cG}_j>0$ for any $j\in \{1,\dots, |S|\}$. 
This unique measure $\xi_\cG$ is refereed to as the \emph{invariant measure} of the stochastic graph $\cG$. We recall the following important properties that we use in what follows.
\begin{lemma}\label{lemma:Ergodicity}
Consider a stochastic graph $\cG$ on $\M$ satisfying Assumption~\ref{assum:StrongConn}. Then, the measure $\mu_{\cG, \xi_\cG}$ on $(\Sigma_M,\cB(\Sigma_M))$ is shift-invariant and ergodic and, for any $\xi \in \Xi_S$, the measure $\mu_{\cG,\xi}$ is absolutely continuous\footnote{Given $2$ measures $\mu,\nu$ (on a generic measurable space), $\mu$ is absolutely continuous with respect to $\nu$ is $\nu(C)=0$ implies $\mu(C)=0$.} with respect to $\mu_{\cG, \xi_{\cG}}$. 
\end{lemma}
The first part holds by ergodic theory, see for example~\cite[Chapter 1]{Walters2000}, while the absolute continuity of $\mu_{\cG,\xi}$ for any $\xi \in \Xi_S$ follows by the fact that $\xi_{\cG}>0$ (component-wise). Now that the stochastic setting is well defined, we introduce the class of systems we study in what follows. 
\begin{defn}[(Stochastic) Switched Systems]\label{defn:StochasticJumpLinearSystems}
Let us consider $M, n\in \N$, $\cA:=\{A_1,\dots,A_M\}\subset \R^{n\times n}$, and a stochastic graph $\cG$ on $\M$. We consider the \emph{discrete-time switched system} $\cS(\cA,\cG)$ defined by
\begin{equation}\label{eq:StocSysDefn}
x(k+1)=A_{\sigma_k} \,x(k),\;\;\;k\in \N,
\end{equation}
where $\sigma\in \Sigma_M$ is also called the \emph{switching sequence}. Given $\xi\in \Xi_S$, we consider the probability space $(\Sigma_M,\cB(\Sigma_M),\mu_{\cG,\xi})$ and the asymptotic behavior of systems~\eqref{eq:StocSysDefn} is then \emph{studied with respect to this measure}.

\end{defn}
Given $x_0\in \R^n$ and $\sigma\in \Sigma_M$, we denote by $x(k,x_0,\sigma)$ the \emph{solution} of~\eqref{eq:StocSysDefn}, starting at $x_0$ with respect to the signal $\sigma$, evaluated at time $k\in \N$. Similarly, given $\cA=\{A_1,\dots, A_M\}\subset \R^{n\times n}$, for any $k\in \N$, given $\wi=(i_{k-1},\dots, i_0)\in \M^k$ we use the notation $A(\wi)=A_{i_{k-1}}\cdot\cdot\cdot A_{i_0}$, and given $\sigma\in \Sigma_M$, $A^k(\sigma)=A_{\sigma_{k-1}}\cdot\cdot\cdot A_{\sigma_0}$.

\section{Almost Sure Stability and probabilistic spectral radius}\label{Sec:AlmostSUre}

We now introduce the considered stability notion and the corresponding spectral characterization.
\begin{defn}
Consider $M,n\in \N$, $\cA=\{A_1,\dots, A_M\}\subset \R^{n\times n}$ and a stochastic graph $\cG$ on $\M$. Given  $\Phi\subseteq \Xi_S$ the switched  system~\eqref{eq:StocSysDefn} is said to be \emph{uniformly almost surely asymptotically stable with respect to $\Phi$}  if, for any $x_0\in \R^n$ and any $\xi\in \Phi$ we have
\begin{equation}
\mu_{\cG,\xi}\left(\left\{\sigma\in \Sigma_M\;\vert\;\;\lim_{k\to \infty}|x(k,x_0,\sigma)|=0\right\}\right)=1.
\end{equation}
In the case  $\Phi=\Xi_S$, the term \emph{\virgolette{with respect to $\Phi$}} is omitted. 
 \end{defn}
 It turns out that this stability notion, (as the ``deterministic'' one, introduced in~\cite{Pro11}) can be characterized by studying a corresponding probabilistic spectral radius.

\begin{defn}[\mat{Probabilistic} Spectral Radius]
Consider $M,n\in \N$, $\cA=\{A_1,\dots, A_M\}\subset \R^{n\times n}$, a stochastic graph $\cG$ on $\M$  and $\xi\in \Xi_S$.
Given any operator matrix norm $\|\cdot\|$, we define
\begin{equation}
\begin{aligned}
\rho_0(\cA,\cG,\xi):=\limsup_{k\to \infty}\left[\prod_{\wi \in \M^k} \|A(\wi)\|^{\mu_{\cG,\xi}(\wi)\,}\right]^\frac{1}{k}.
\end{aligned}\label{eq:AlmostSureJSR}
\end{equation}
The quantity  $\rho_0(\cA,\cG,\xi)$ is referred to as the \emph{probabilistic spectral radius of $\cA$ induced by $\cG$ and $\xi\in \Xi_S$}.
\end{defn}
\mat{The subscript $0$ in $\rho_0(\cA,\cG,\xi)$ is motivated by the fact that this radius can be seen as the limit, for $q$ going to $0$, of the \emph{$q$-spectral radius}, i.e. the quantity characterizing the stability of the $q$-moment of solutions, see~\cite{Pro08,FangLop1995}.}
From now on we study stability with respect to the whole set $\Xi_S$; under Assumption~\ref{assum:StrongConn} this is not restrictive: using Lemma~\ref{lemma:Ergodicity} it can be shown that, for any $\xi \in \Xi_S$, we have 
\begin{equation}\label{eq:SupremumProbability}
\sup_{\xi\in \Xi_S}\rho_0(\cA,\cG,\xi)=\rho_0(\cA,\cG,\xi_\cG),
\end{equation}
see \cite[Lemma 2.3]{FangLop1995}. We have the following relation between almost sure stability and probabilistic spectral radius.
\begin{prop}
Consider $M,n\in \N$, $\cA=\{A_1,\dots, A_M\}\subset \R^{n\times n}$ and a stochastic graph $\cG$ on $\M$ satisfying Assumption~\ref{assum:StrongConn}. Then, the system $\cS(\cA,\cG)$ is uniformly almost surely asymptotically stable if and only if
\[
\rho_0(\cA,\cG,\xi_\cG)<1.
\] 
\end{prop}
\begin{proof}[Sketch of the Proof]
The proof can be found in~\cite{FangLop1995} for the case of Markov jump linear systems (i.e. for stochastic graphs arising from strongly connected Markov chains), and the ideas can be adapted in this context, \emph{mutatis mutandis}. The peculiarity here is that we consider the probabilistic spectral radius as defined in~\eqref{eq:AlmostSureJSR}. Instead, most of the literature concerning stability of stochastic switched systems (cfr.~\cite{FangLop1995,Fang97,ProJun13,SutFawRen21} and references therein) introduce the \emph{(maximal) Lyapunov Exponent}  
\begin{equation}\label{eq:LyapExp}
\lambda_0(\cA,\cG,\xi):=\limsup_{k\to \infty}\frac{1}{k}\E_{\cG,\xi}\left(\log \,\|A^k(\sigma)\|\right)
\end{equation}
to characterize almost sure stability. Since it holds that $e^{\lambda_0(\cA,\cG,\xi)}=\rho_0(\cA,\cG,\xi)$, we can apply the same arguments in our case. 
\end{proof}
\mat{The reason to consider the probabilistic spectral radius (instead of the Lyapunov exponent) is two-fold: from one side, this allows  to draw a parallel with the definition and estimation techniques for the JSR (see~\cite{ChiMaz21}), and, on the other hand, it allows us to simplify the following notation and proofs.}
\subsection{Lyapunov Multi Functions for Almost-Sure Stability}

To give a \virgolette{computable} characterization of the probabilistic spectral radius, we define the space of candidate Lyapunov functions.
\begin{defn}
We say that $f$ is a \emph{candidate Lyapunov function} (and we write $f\in \cF_n$) if 
 $f:\R^n\to \R$ is continuous, positive definite, and positively homogeneous\footnote{A function $f:\R^n\to \R$ is positively homogeneous if $f(ax)=af(x)$, for any $a\in \R_+$ and any $x\in \R^n$.}.
\end{defn}
\begin{defn}[Lyapunov Multi-Functions]
Consider $\cA=\{A_1,\dots, A_N\}\subset \R^{n\times n}$, a stochastic graph $\cG=(S,E,p)$ and a scalar $\rho\geq 0$. A \emph{Lyapunov multi function (LMF) for $\cS(\cA,\cG)$ w.r.t. $\rho$} is a set of $|S|$ functions $F:=\{f_a\in \cF_n\,\;\;\vert\; a\in S\}$ such that, $\forall x\in \R^n, \forall a\in S$, 
\begin{equation}\label{eq:DefinitionLyapunovFunctionals}
\prod_{i\in \M}\prod_{b\in S}(f_b(A_ix))^{p_{a,b,i}}\leq\rho f_a(x).
\end{equation}
\end{defn}
\begin{prop}\label{prop:SpectralRadiiCharachterization}
Consider $\cA=\{A_1,\dots, A_N\}\subset \R^{n\times n}$, a stochastic graph $\cG$, it holds that
\begin{equation}\label{eq:InfimumJSRInequality}
\begin{aligned}
&\sup_{\xi\in \Xi_S}\rho_0(\cA,\cG,\xi)\leq\\&\inf\left\{\rho\geq 0\,\vert\,\exists\;F\subset\cF_n, \,\text{LMF for $\cS(\cA,\cG)$ w.r.t. $\rho$}\right\}.
\end{aligned}
\end{equation}
If $\cG$ satisfies Assumption~\ref{assum:StrongConn} the equality holds, i.e.,
\begin{equation}\label{eq:InfimumJSREquality}
\begin{aligned}
&\rho_0(\cA,\cG,\xi_\cG)=\\&\inf\left\{\rho\geq 0\,\vert\,\exists\;F\subset\cF_n, \,\text{LMF for $\cS(\cA,\cG)$ w.r.t. $\rho$}\right\}.
\end{aligned}
\end{equation}
\end{prop}
\begin{proof}
First, consider a Lyapunov multi function for $\cS=(\cA,\cG)$ w.r.t. $\rho$, denoted by $F:=\{f_s\in \cF_n,\vert\;s\in S\}$. We show that for every $\xi\in \Xi_S$, we have $\rho_0(\cA,\cG,\xi)\leq \rho$.
 By finiteness of $S$ and since $f_s\in \cF_n$ are homogeneous of degree~$1$, continuous and positive definite, there exist $\alpha,\beta\in \R_+$ such that
 \begin{equation}\label{eq:SandwichInequality}
 \alpha |x|\leq f_s(x)\leq \beta|x|\;\;\;\forall x\in \R^n,\;\;\forall s\in S.
 \end{equation} 
Let us consider any $\xi\in \Xi_S$, we want to prove that, for every $k\in \N$,
\begin{equation}\label{eq:MainInequalityProof2}
\prod_{s\in S}\prod_{\wj\in \M^k}f_s(A(\wj)x)^{\bP_{\cG,\xi}(s,\wj)}\leq \rho^k\prod_{s\in S}f_s(x)^{\xi(s)}, \;\forall\;x\in \R^n.
\end{equation}
We prove it by induction, the case $k=0$ is a trivial identity (posing, by definition $\M^0:=\{\emptyset\}$). We consider any $k\in \N$ and supposing that~\eqref{eq:MainInequalityProof2} is verified for $k-1$ we prove that it is also true for $k$. Computing 
\[
\begin{aligned}
&\prod_{s\in S}\prod_{\wj\in \M^k}f_s(A(\widehat j)x)^{\bP_{\cG,\xi}(s,\wj)}=\prod_{s\in S}\prod_{\wj\in \M^k}f_s(A(\wj)x)^{\sum_{a\in S} \bP_{\cG,\xi}(a,\wj^-)p_{a,s,\wj_f}}\\\\
=&\prod_{s\in S}\prod_{\wj\in \M^k}\prod_{a\in S}\left [f_s(A(\wj)x)^{p_{a,s,\wj_f}} \right ]^{ \bP_{\cG,\xi}(a,\wj^-)}=\prod_{a\in S}\prod_{\wi\in \M^{k-1}}\left[\prod_{s\in S}\prod_{j\in \M} f_s(A_jA(\wi)x)^{p_{a,s,j}} \right ]^{ \bP_{\cG,\xi}(a,\wi)}
\\\\\leq &\rho\prod_{a\in S}\prod_{\wi\in \M^{k-1}} f_a(A(\wi)x)^{ \bP_{\cG,\xi}(a,\wi)}\leq \rho \rho^{k-1}\prod_{s\in S}f_s(x)^{\xi(s)}=\rho^{k}\prod_{s\in S}f_s(x)^{\xi(s)}.
\end{aligned}
\]
Now, using~\eqref{eq:SandwichInequality} and~\eqref{eq:MainInequalityProof2} we obtain
\[
\begin{aligned}
&\prod_{\wj\in \M^k} \|A(\wj)\|^{\mu_{\cG,\xi}(\wj)\,}=\prod_{s\in S}\prod_{\wj\in \M^k} \max_{x\neq 0}\frac{|A(\wj)x|}{|x|}^{\bP_{\cG,\xi}(s,\wj)}=\prod_{s\in S}\prod_{\wj\in \M^k} \left ( \max_{x\neq 0}\frac{|A(\wj)x|}{\prod_{a\in S}|x|^{\xi(a)}}\right )^{\bP_{\cG,\xi}(s,\wj)}\\\leq&\frac{\beta}{\alpha} \prod_{s\in S}\prod_{\wj\in \M^k} \left ( \max_{x\neq 0}\frac{f_s(A(\wj)x)}{\prod_{a\in S}f_a(x)^{\xi(a)}}\right )^{\bP_{\cG,\xi}(s,\wj)}\leq \frac{\beta}{\alpha}\rho^{k}.
\end{aligned}
\]
Concluding we have, for any $\xi \in \Xi_S$
$
\rho_0(\cA,\cG,\xi):=\limsup_{k\to \infty}\left[\prod_{\wj\in \M^k} \|A(\wj)\|^{\mu_{\cG,\xi}(\wj)\,}\right]^\frac{1}{k}\leq \lim_{k\to \infty} \left ( \frac{\beta}{\alpha}\right)^{\frac{1}{k}}\rho=\rho$,
as to be proven.

The idea of the proof of~\eqref{eq:InfimumJSREquality} can be found in~\cite{Pro10} for the case of i.i.d. matrices; the construction applies in the ergodic case (i.e. under Assumption~\ref{assum:StrongConn} and applying~\eqref{eq:SupremumProbability}), \emph{mutatis mutandis}.  
\end{proof}

In Proposition~\ref{prop:SpectralRadiiCharachterization} we consider the infimum over \emph{all} the possible candidate Lyapunov functions in $(\cF_n)^{|S|}$. Since, usually, it is practical to restrict the search to a particular subset (e.g. quadratic norms, SOS-polynomials, etcetera), we provide the following definition.

\begin{defn}\label{defn:TemplateConstrainedRadius}
Consider $\cA=\{A_1,\dots, A_N\}\subset \R^{n\times n}$ and a stochastic graph $\cG$. Consider a subclass of candidate Lyapunov functions $\cV\subset \cF_n$, we define
\begin{equation}\label{eq:TemmplateConstRadius}
\begin{aligned}
&\rho_{0,\cV}(\cA,\cG):=\\&\;\;\;\inf\left\{\rho>0\;\vert\;\exists\;F\subset\cV \;\;\text{LMF for $\cS(\cA,\cG)$ w.r.t. $\rho$}\right\}.
\end{aligned}
\end{equation}
\end{defn}

In Proposition~\ref{prop:SpectralRadiiCharachterization} we have proven that 
$\rho_{0,\cF_n}(\cG,\cA)=\sup_{\xi\in \Xi_S}\rho_0(\cG,\cA,\xi)$, and thus, for any $\cV\subset \cF_n$,
\begin{equation}\label{eq:ApproximationGivenFamily}
\sup_{\xi\in \Xi_S}\rho_0(\cG,\cA,\xi)\leq\rho_{0,\cV}(\cG,\cA).
\end{equation}
\section{Approximation of the probabilistic spectral radius: Graph-Based Lifts}\label{sec:LIFt}
In the following, we define some ``expansion techniques''  of stochastic switched systems, providing related results concerning the computation of the probabilistic spectral radius.

\vspace{-0.5cm}
\subsection{The Step Lift}
We introduce a formal expansion of stochastic graph which, intuitively, induces the same stochastic framework on $\Sigma_M$, while focusing not on the transition among \emph{letters} (i.e. $i\in \M$), but among sub-words of the form $\wi =(i_{K-1},\,\dots, \,i_0)\in \M^K$  of arbitrary length $K\in \N$. 
\begin{defn}[The Step Lift]\label{defn:TStepsLift}
Let us consider $M\in \N$, a stochastic graph $\cG=(S,E,p)$ on $\M$ and a set of matrices $\cA=\{A_1,\dots,A_M\}$. Given the system $\cS(\cA,\cG)$ and any integer $K\geq 1$, the \emph{$K$-step lift of $\cS(\cA,\cG)$} denoted by $\cL^K\cS(\cA,\cG)$ is a stochastic system defined by a stochastic graph on $\M^K$, $\cG^K=(S^K, E^K, p^K)$ and a set of matrices $\cA^K$ defined as follows:
\begin{enumerate}[leftmargin=*]
\item $S^K\equiv S$,
\item For any \virgolette{candidate edge} for $\cG^K$, $(a,b,\wi)\in S\times S \times \M^K$, we  inductively define its probability weight by
\[
p^K_{a,b,\wi}:=\sum_{c\in S} p^{K-1}_{a,c,\wi^-}\;p_{c,b,\wi_f}.
\]
By convention, if $p_{a,b,\wi}=0$ then $e=(a,b,\wi)\notin E^K$.\label{Item:DefnProbability}
\item $\cA^K:=\{A({\wi})\;\vert\;\wi\in \M^K\}$, where, given $\wi=(i_{K-1},\dots i_0)$ we recall that $A(\wi)=A_{i_{K-1}}\cdot\cdot \cdot A_{i_0}$.
\end{enumerate}
\end{defn}
It is clear that $\cL^1\,\cS(\cA,\cG)=\cS(\cA,\cG)$. In the following statement we collect the relations between the probability measures induced by $\cG$ and $\cG^K$, respectively.
\begin{lemma}\label{lemma:PropSimpleStepLIft}
Consider a stochastic graph $\cG$ on $\M$ and $K\in \N$; it holds that
$
P_{\cG^K}=P_{\cG}^K$, 
where $P_\cG$ and $P_{\cG^K}$ are defined as in~\eqref{eq:DefnMatrixAssociatedtoG}; this  implies $\xi_\cG=\xi_{\cG^K}$. Moreover, for any $k\in \N$, $K\in \N\setminus\{0\}$, any $\xi\in \Xi_S$ and any $\wi\in \M^{kK}$ we have
\begin{equation}\label{eq:MainPropertyTlifts}
\mu_{\cG,\xi}(\wi)=\mu_{\cG^K,\xi}(\wi).
\end{equation}
\end{lemma}
\begin{proof}
The first part follows from the definition in~\eqref{eq:DefnMatrixAssociatedtoG} and from Item~\ref{Item:DefnProbability} of Definition~\ref{defn:TStepsLift}. Equation~\eqref{eq:MainPropertyTlifts} is a consequence of $P_{\cG^K}=P^{K}_\cG$, once recalled~\eqref{eq:ProbabilityDefns} and~\eqref{eq:ProbabilityDefnsPaths}.
\end{proof}
\begin{figure*}[t!]
\begin{subfigure}{0.3\linewidth}
  \centering
\begin{tikzpicture}%
  [>=stealth,
   shorten >=1pt,
   node distance=1.3cm,
   on grid,
   auto,
   every state/.style={draw=red!60, fill=red!5, thick}
  ]
\node[state,inner sep=1pt, minimum size=22pt] (left)                  {$a$};
\node[state,inner sep=1pt, minimum size=22pt] (right) [right=of left] {$b$};
\path[->]
   (left) edge[bend left=30]     node          [scale=1]            {$(2,\frac{2}{3})$} (right)
        (right)   edge[bend left=30] node        [scale=1]        {$(1,\frac{1}{4})$} (left)
   (left) edge[loop left=60]     node            [scale=1]          {$(1,\frac{1}{3})$} (left)
   (right) edge[loop right=60]     node      [scale=1]                {$(2,\frac{3}{4})$} (right)
   ;
\end{tikzpicture}
  \caption{The graph $\cH$ in Example~\ref{ex:GraphAndKlift}.}
  \label{Fig:GraphH}
  \end{subfigure}
  \begin{subfigure}{0.3\linewidth}
  \centering
\begin{tikzpicture}%
  [>=stealth,
   shorten >=1pt,
   node distance=3.2cm,
   on grid,
   auto,
   every state/.style={draw=red!60, fill=red!5, thick}
  ]
\node[state, inner sep=1pt, minimum size=20pt] (left)                  {$a$};
\node[state, inner sep=1pt, minimum size=20pt] (right) [right=of left] {$b$};
\path[->]
   (left) edge[bend left=50]     node              [scale=1]            {$(2\,1,\frac{2}{9})$} (right)
   (left) edge[bend left=10]     node             [scale=1]             {$(2\,2,\frac{1}{2})$} (right)
        (right)   edge[bend left=50] node          [scale=1]          {$(1\,1,\frac{1}{12})$} (left)
         (right)   edge[bend left=10] node            [scale=1]        {$(1\,2,\frac{3}{16})$} (left)
   (left) edge[out=130, in=90, looseness=8]   node                [above,scale=1]          {$(1\,1,\frac{1}{9})$} (left)
   (left) edge[out=230, in=270, looseness=8]   node [below, scale=1]           {$(1\,2,\frac{1}{6})$} (left)
   (right) edge[out=90, in=50, looseness=8]     node            [above,scale=1]           {$(2\,2,\frac{9}{16})$} (right)
   (right) edge[out=270, in=310, looseness=8]     node       [below, scale=1]    { $(2\,1,\frac{1}{6})$} (right)
   ;
\end{tikzpicture}
  \caption{The $2$-Step lift of $\cH$, $\cH^2$.}
\label{Fig:StpesLift}
  \end{subfigure}
  \begin{subfigure}{0.33\linewidth}
  \centering
\begin{tikzpicture}%
  [>=stealth,
  shorten >=1pt,
  node distance=1cm,
  on grid,
  auto,
  every state/.style={draw=red!60, fill=red!5, thick}
  ]
  \node[state, inner sep=1pt, minimum size=20pt] (left)                  {$aa$};
  \node[state, inner sep=1pt, minimum size=20pt] (right) [right=of left, xshift=2.5cm] {$bb$};
  \node[state, inner sep=1pt, minimum size=20pt] (upper) [above right=of left, xshift=1.1cm, yshift=0.05cm]{$ab$};
  \node[state, inner sep=1pt, minimum size=20pt] (below) [below right=of left, xshift=1.1cm, yshift=0cm]{$ba$};
  \path[->]
  (left) edge[loop left=60]     node             [scale=0.9]             {$(1,\frac{1}{3})$} (left)
  (left) edge[bend left=15]     node                      {$(2,\frac{2}{3})$} (upper)
  (upper) edge[bend left=15]     node                      {$(1,\frac{1}{4})$} (below)
  (below) edge[bend left=15]     node                      {$(1,\frac{1}{3})$} (left)
  (below) edge[bend left=15]     node                      {$(2,\frac{2}{3})$} (upper)
  (right) edge[bend left=15]     node                      {$(1,\frac{1}{4})$} (below)
  (upper) edge[bend left=15]     node                      {$(2,\frac{3}{4})$} (right)
  (right) edge[loop right=300]     node            [scale=0.9]              {$(2,\frac{3}{4})$} (right)
;
  \end{tikzpicture}
  \caption{ The Path lift of degree $1$, $\mathcal{H}_1$ of the graph $\cH$.}
  \label{fig:PathDependLift}
\end{subfigure}
\label{Fig:ThereeLifts}
  \caption{A stochastic graph and its $2$-Step Lift and Path Lift of degree $1$.}
\end{figure*}
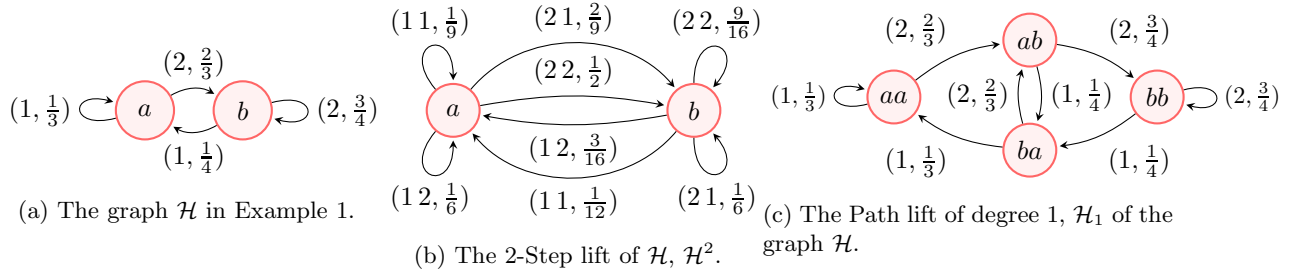

\begin{example}\label{ex:GraphAndKlift}
Consider the stochastic graph $\cH$  in Figure~\ref{Fig:GraphH}. It represents the case of a (strongly connected) Markov chain, with transition matrix given by $P=P_{\cH}=\left [\begin{smallmatrix} 1/3 & \;2/3 \\ 1/4 & \;3/4 \end{smallmatrix}\right]$ and $\xi_\cH=[3/11, \;8/11]^\top$. The corresponding step lift of degree $2$, denoted by $\cH^2$ is depicted in Figure~\ref{Fig:StpesLift}. Recalling the definition of $P_\cG$ given by~\eqref{eq:DefnMatrixAssociatedtoG} is easy to see that
$
P_{\cH^2}=\left [\begin{smallmatrix} 5/18 &13/18\\ 13/48& 35/48   \end{smallmatrix}\right]=P_\cH^2$,
as predicted by Lemma~\ref{lemma:PropSimpleStepLIft}.
\end{example}
\begin{thm}[Properties of Step Lift]\label{thm:TLIFT}
Consider $M\in \N$, a stochastic graph $\cG$ on $\M$ and $\cA=\{A_1,\dots,A_M\}\subset \R^{n\times n}$. For any $K\in \N\setminus\{0\}$, and any $\xi\in \Xi_S$ it holds that
\begin{equation}\label{eq:firstInequalityTstepseq}
\rho_0(\cA^K,\cG^K,\xi)=\left [\rho_0(\cA,\cG,\xi)\right]^K.
\end{equation}
For any (non-empty) subclass of candidate functions $\cV\subset \cF_n$, we have
\begin{equation}\label{eq:Kliftapproxiamtion}
\sup_{\xi\in \Xi}\rho_0(\cA,\cG,\xi)\leq \sqrt[K]{\rho_{0,\cV}(\cA^K,\cG^K)}\leq  \rho_{0,\cV}(\cA,\cG),
\end{equation}
and moreover, if $\cG$ satisfies Assumption~1, we have
\begin{equation}\label{eq:LimitKSteps}
\rho_0(\cA,\cG,\xi_\cG)=\lim_{K\to +\infty}\sqrt[K]{\rho_{0,\cV}(\cA^K,\cG^K)}.
\end{equation}
\end{thm}
\begin{proof}
The equivalence in~\eqref{eq:firstInequalityTstepseq} follows by~\eqref{eq:MainPropertyTlifts}. Indeed, consider any $K\in \N$ and any $\xi\in \Xi_S$ computing,
\[
\begin{aligned}
&\rho_0(\cA,\cG,\xi)=\limsup_{k\to \infty}\left[\prod_{\wi \in \M^k} \|A(\wi)\|^{\mu_{\cG,\xi}(\wi)\,}\right]^\frac{1}{k}= \left [\limsup_{k'\to \infty}\left[\prod_{\wj \in (\M^K)^{k'}} \|A(\wj)\|^{\mu_{\cG^K,\xi}(\wj)\,}\right]^\frac{1}{k'}\right]^{\frac{1}{K}}=\rho_0(\cA^K,\cG^K,\xi)^{\frac{1}{K}}.
\end{aligned}
\] 
For~\eqref{eq:Kliftapproxiamtion}, the left-inequality follows by~\eqref{eq:firstInequalityTstepseq} and~\eqref{eq:ApproximationGivenFamily}. For the right inequality in~\eqref{eq:Kliftapproxiamtion}, let us consider any $\wt \rho > \rho_{0,\cV}(\cA,\cG)$. We recall that, by definition, there exists $F:=\{f_s\;\;\vert\;s\in S\}\in \cV^{|S|}$, such that
\[
\prod_{i\in \M}\prod_{b\in S}f_b(A_ix)^{p_{a,b,i}}\leq \wt \rho f_a(x),\;\;\forall x\in \R^n,\;\forall a\in S.
\]
Firstly, we  prove by induction that, $\forall \,K\in \N\setminus\{0\}$, 
\begin{equation}\label{eq:TListsInductiveStep}
\prod_{\wi\in \M^K}\prod_{b\in S}f_b(A(\wi)x)^{p_{a,b, \wi}}\leq \wt \rho ^K f_a(x),\;\;\forall x\in \R^n,\;\forall a\in S.
\end{equation}
The case $K=1$ is trivial, let us suppose the statement is true for $K-1$, computing, for any $a\in A$ and any $x\in \R^n$ we have
\[
\begin{aligned}
&\prod_{\wi\in \M^K}\prod_{b\in S}f_b(A(\wi)x)^{p^K_{a,b, \wi}}
=\prod_{\wi\in \M^K}\prod_{b\in S}f_b(A(\wi)x)^{\sum_{c\in S}p^{K-1}_{a,c, \wi^-}p_{c,b,\wi_f}}\\
=&\prod_{\widehat j\in \M^{K-1}}\prod_{c\in S}\left [\prod_{h\in \M}\prod_{b\in S}f_b(A_hA(\widehat j)x)^{p_{c,b,h}}\right ]^{p^{K-1}_{a,c, \widehat j}}\leq \wt \rho \prod_{\widehat j\in \M^{K-1}}\prod_{c\in S}f_c(A(\widehat j\,)x)^{p^{K-1}_{a,c, \widehat j}}
\leq  \wt \rho ^{K}f_a(x).
\end{aligned}
\]
We have thus obtained that for any $\wt \rho> \rho_{0,\cV}(\cA,\cG)$ there exist a set $F=\{f_s,\;\vert\;s\in S\}\in \cV^{|S|}$ such that~\eqref{eq:TListsInductiveStep} holds, and recalling Definition~\ref{defn:TemplateConstrainedRadius}, we have $\rho_{0,\cV}(\cA^K,\cG^K)< \wt \rho ^K$ and thus $\rho_{0,\cV}(\cA^K, \cG^K)\leq \rho_{0,\cV}(\cA,\cG)^K$, as to be proven.
We now prove~\eqref{eq:LimitKSteps}; we denote  $\rho:=\rho_0(\cA,\cG,\xi_\cG)$.   Let us consider any $f\in\cV\subset\cF_n$ by homogeneity there exist $0<\alpha\leq \beta$ such that
\begin{equation}\label{eq:Sandwitch2}
\alpha |x|\leq f(x)\leq \beta |x|,\;\;\;\forall \;x\in \R^n.
\end{equation}
For any $a\in S$, we will denote by $\xi_a\in \Xi_S$ the indicator measure of $a\in S$, i.e. the probability measure defined by $\xi_a(a)=1$ and $\xi_a(s)=0$ for all $s\in S$, $s\neq a$. 
By~\eqref{eq:AlmostSureJSR} and~\eqref{eq:SupremumProbability}, for any $\varepsilon_1>0$ there exists a $K_1\in \N$ such that
\[
\left [\prod_{\wi \in \M^{K_1}} \|A(\wi)\|^{\mu_{\cG,\xi_a}(\wi)\,}\right]^\frac{1}{K_1}\leq \rho+\varepsilon_1,\;\;\;\; \forall a\in S.
\]
Recalling~\eqref{eq:DefProbabilityCountingPaths}, it is clear that $\mu_{\cG,\xi_a}(\wi)=\sum_{b\in S}p^{K_1}_{a,b,\wi}$, and we can thus write
\[
\left [\prod_{\wi \in \M^{K_1}}\prod_{b\in S} \|A(\wi)\|^{p^{K_1}_{a,b,\wi}}\right]^\frac{1}{K_1}\leq \rho+\varepsilon_1,\;\;\;\; \forall a\in S.
\]
Developing
\[
\begin{aligned}
\left [\prod_{\wi \in \M^{K_1}}\prod_{b\in S} \left(\max_{x\neq 0}\frac{|A(\wi)x|}{|x|}\right)^{p^{K_1}_{a,b,\wi}}\right]^\frac{1}{K_1}\leq \rho+\varepsilon_1,\;\;\;\forall\,a\in S,
\end{aligned}
\]
implies that, for all $x\in \R^n$, we have
\[
\left [\prod_{\wi \in \M^{K_1}}\prod_{b\in S} |A(\wi)x|^{p^{K_1}_{a,b,\wi}}\right]^\frac{1}{K_1}\leq (\rho+\varepsilon_1)\,|x|^{\frac{1}{K_1}},\,\forall \,a\in S.
\]
Now, identifying $f_s\equiv f$, for all $s\in S$ and using~\eqref{eq:Sandwitch2} we obtain that, for all $x\in \R^n$,
\[
\prod_{\wi \in \M^{K_1}}\prod_{b\in S} f_b(A(\wi)x)^{p^{K_1}_{a,b,\wi}}\leq\frac{\beta}{\alpha}(\rho+\varepsilon_1)^{K_1}\,f_a(x),\;\;\forall \,a\in S.
\]
Recalling the definition in~\eqref{eq:TemmplateConstRadius} and by induction step in the proof of~\eqref{eq:Kliftapproxiamtion}, we have thus proven that, for every $\varepsilon>0$, there exists a $K_1\in \N$ such that $\rho_\cV(\cG^K,\cA^K)\leq \frac{\beta}{\alpha}(\rho+\varepsilon_1)^{K}$, for every $K\geq K_1$. 
Now, given any $\varepsilon>0$, it suffices to choose $\varepsilon_1<\varepsilon$ small enough and $K\in \N$, $K\geq K_1$ big enough  such that $\frac{\beta}{\alpha}(\rho+\varepsilon_1)^{K}\leq (\rho+\varepsilon)^{K}$. This implies that, $\forall\varepsilon>0$ there exists $K\in \N$ such that
\[
\rho_\cV(\cG^K,\cA^K)\leq (\rho+\varepsilon)^{K}
\]
and thus $\lim_{K\to \infty}\rho_\cV(\cG^K,\cA^K)=\rho^K=\rho_0(\cG,\cA,\xi_\cG)^K$, concluding the proof.
\end{proof}

 We note here that in Lemma~\ref{lemma:PropSimpleStepLIft} and Theorem~\ref{thm:TLIFT}, Assumption~\ref{assum:StrongConn} is not required, except for the convergence property in~\eqref{eq:LimitKSteps}. In other words, since the $K$-step lift conserves the probabilistic structure for any $\cG$ and with respect to any initial probability $\xi\in \Xi_S$, the ergodicity is not required. 
\begin{rmk}
Theorem~\ref{thm:TLIFT} provides a technique to approximate the probabilistic spectral radius. Once a ``practical'' (in terms of optimization purposes) family of functions is fixed, choose a $K\in \N$, construct the $K$-step lift of the considered system $\cS(\cA,\cG)$. Then, solve (or approximate) the optimization problem in~\eqref{eq:TemmplateConstRadius} and find (an upper bound for) $\rho_{0,\cV}(\cA^K,\cG^K)$. The $K$-root, $\rho_{0,\cV}(\cA^K,\cG^K)^{\frac{1}{K}}$ by Theorem~\ref{thm:TLIFT} provides an upper bound for the probabilistic spectral radius of $\cS(\cA,\cG)$. Under Assumption~\ref{assum:StrongConn} this upper bound also converges to $\rho_0(\cA,\cG,\xi_\cG)$ \emph{no matter} the chosen family $\cV$; of course, the ``speed of convergence'' will depend on $\cV$.
\end{rmk}
\subsection{The Path Lift}
In this subsection we propose another lift, defining an augmented graph which, intuitively, adds memory to the framework, considering paths of given length as new states. 
\begin{defn}[The Path Lift]\label{defn:RPathss}
Consider $M\in \N$, a stochastic graph $\cG=(S,E,p)$ on $\M$ and $\cA=\{A_1,\dots,A_M\}\subset\R^{n\times n}$. Given any integer $R\geq 1$, the \emph{path lift of degree $R$ of $\cS(\cA,\cG)$} denoted by $\cL_R \cS(\cA,\cG)$ is a stochastic system composed by a stochastic graph on $\M$, $\cG_R=(S_R,E_R,p_R)$ defined recursively as follows:
\begin{enumerate}[leftmargin=*]
\item For any path of length $R$, $\overline q=(e_1,\dots, e_R)\in \Pt^R(\cG)$ in $\cG$,  add a node $s_{\overline q}\in S_R$;
\item For each path $\overline r=(e_1,e_2,\dots e_R,e_{R+1})$ of length $R+1$ in $\cG$, with $i\in \M$ the label of $e_{R+1}$, add in $E_R$ the edge $(s_{\overline{q}_1},s_{\overline{q}_2},i)$, where $\overline{q}_1=(e_1,\dots, e_R)$ and $\overline{q}_2=(e_2,\dots, e_{R+1})$;
\item For any path $\overline r=(e_1,e_2,\dots e_R,e_{R+1})\in E_R$ of length $R+1$, set $p_R(\overline r)=p(e_{R+1})$.\label{Item:Item3DefnPathLifts}
\end{enumerate}
\end{defn}
Given any $\xi\in \Xi_S$, we consider \emph{the path lift measure of degree $R$ $\xi_R\in \Xi_{S_R}$} defined as follows: for every $s_{\overline q}\in S_R$ with $\overline q=(e_1,\dots, e_R)$ and $e_1=(a,b,i)\in E$, define:
\begin{equation}\label{eq:RLiftMeasure}
\xi_R(s_{\overline q}):=\xi(a)p(e_1)\cdots p(e_R).
\end{equation}
From Definition~\ref{defn:RPathss} we have that, for any $R\in \N$, $\cG_{R+1}=(\cG_R)_1$ and, similarly, for any $\xi\in \Xi_S$, $\xi_{R+1}=(\xi_R)_1$ i.e. the path lift of degree $R+1$ is the path lift (of degree $1$) of the path lift of degree $R$.
It can be seen that, given any stochastic graph $\cG$ and any $\xi\in \Xi_S$ and any $R\in \N$, $\cG_R$ and $\xi_R\in \Xi_{S_R}$ introduced in Definition~\ref{defn:RPathss} are a well-defined stochastic graph and a probability measure, respectively. Moreover, if $\cG$ satisfies Assumption~\ref{assum:StrongConn}, so does $\cG_R$.
\begin{lemma}\label{lem:LyapunovMeasureRLift}
Consider a stochastic graph $\cG$ satisfying Assumption~\ref{assum:StrongConn}, and consider $\xi_\cG$ its invariant measure. Then, for every $R\in \N$, we have that
$
(\xi_\cG)_R=\xi_{\cG_R}$, 
i.e. the path lift measure of degree $R$ of $\xi_\cG$ is the invariant measure of $\cG_R$.
\end{lemma}
\begin{proof}
Consider any stochastic graph $\cG=(S,E,p)$.
Since, for any $R\in \N$ we have that $\cG_{R+1}=(\cG_R)_1$ it suffices to prove the claim for $R=1$.
We want to prove that $(\xi_\cG)_1=\xi_{\cG_1}$; since $\xi_\cG$ is the invariant measure of $\cG$, recalling~\eqref{eq:DefnMatrixAssociatedtoG}, we have that
$
\xi_\cG(r)=\sum_{s\in S}\xi_\cG(s)\sum_{i\in \M}p_{s,r,i}$, $\forall r\in S$.
Consider any $e=(a,b,i)\in \cG$, we have that
\[
\begin{aligned}
(\xi_\cG)_1&(e)=\xi_\cG(a)p(e)=\sum_{s\in S}\xi_\cG(s)\sum_{i\in \M}p_{s,a,i}\,p(e)\\&=\sum_{s\in S}\sum_{i\in \M}\xi_\cG(s)p_{s,a,i}\,p(e)\\&=\sum_{\substack{f\in E, \\ \text{\emph{end}}(f)= \text{\emph{st}}(e)} }(\xi_{\cG})_1(f)\,p(e)=\sum_{f\in E}(\xi_{\cG})_1(f)\, p_1((f,e)),
\end{aligned}
\] 
recalling that, by Item~\ref{Item:Item3DefnPathLifts} of Definition~\ref{defn:RPathss}, we have $p_1((f,e))=p(e)$.
 We have thus proven that $(\xi_\cG)_1$ is Lyapunov for $\cG_1$, and by uniqueness of invariant measure, we conclude.
\end{proof}
\begin{example}\label{ex:ExamplePathDepLift}
Consider again the stochastic graph $\cH$ in Figure~\ref{Fig:GraphH}. The corresponding path lift of degree $1$, $\cH_1$ is represented in Figure~\ref{fig:PathDependLift}. It can be seen that (considering the lexicographic order on the nodes), we have
\[
P_{\cH_1}=\left [\begin{smallmatrix} 1/3 & \;2/3 & \;0 & \;0 \\ 0 &\; 0& \; 1/4 & \;3/4 \\ 1/3& \;2/3 &\;0 & \;0\\ 0 &\; 0& \; 1/4 & \;3/4 \end{smallmatrix}\right],
\]
and computing, we obtain $\xi_{\cH_1}=[\begin{smallmatrix}1/11 &\; 2/11 &\;2/11 &\;6/11\end{smallmatrix}]^\top$, which is equal to $(\xi_\cH)_1$, as proven in Lemma~\ref{lem:LyapunovMeasureRLift}. 
\end{example}
Now we can prove the main result of this subsection, establishing relations between the probability measure on $(\Sigma_M,\cB(\Sigma_M))$ induced by a stochastic $\cG$ and $\xi\in \Xi_S$ and by $\cG_R$, its path lift of degree $R$, and the corresponding $\xi_R$.
\begin{thm}[Properties path lift of degree $R$]\label{thm:PropertiesPatg}
Consider any stochastic graph $\cG$, any distribution $\xi\in \Xi_S$ and any $R\in \N$.
 For all $k\in \N$, for all $\wi \in \M^k$, we have \label{item:R.1}
\begin{equation}\label{eq:MeasureRPathLift}
\begin{aligned}
\mu_{\cG_R,\xi_R}&(\wi)=\mu_{\cG,\xi} (\ell^{-R}(\,[\wi]))\\&=\sum_{\wj\in \ell^{-R}([\wi])} \mu_{\cG,\xi}(\wj),\;\;\forall \wi\in \M^k,\,\forall k\in \N.
\end{aligned}
\end{equation}
If $\cG$ satisfies Assumption~\ref{assum:StrongConn}, we have \label{item:R.2}
\begin{equation}\label{eq:EquivalenceMeasureUnderErgodicity}
\mu_{\cG_R,\xi_{\cG_R}}=\mu_{\cG,\xi_\cG},
\end{equation}
and, for any  $\cA=\{A_1\dots,A_M\}\subset\R^{n \times n}$ and any $\cV\subset\cF_n$, 
\begin{equation}\label{eq:EquivalencesPathsLiftRadius}
\rho_0(\cA,\cG,\xi_\cG)\hspace{-0.05cm}=\hspace{-0.05cm}\rho_0(\cA,\cG_R,\xi_{\cG_R})\hspace{-0.05cm}\leq \rho_{0,\cV}(\cA,\cG_R)\leq \rho_{0,\cV}(\cA,\cG).
\end{equation}
\end{thm}
\begin{proof}
 Since, we recall, for any $R\in \N$, we have that $\cG_{R+1}=(\cG_R)_1$ and, similarly, for any $\xi\in \Xi_S$, $\xi_{R+1}=(\xi_R)_1$, it suffices to prove the claims for $R=1$. Let us prove~\eqref{eq:MeasureRPathLift};
consider any $k\in \N$ and any $\wi \in \M^k$, recalling~\eqref{eq:DefProbabilityCountingPaths} and Definition~\ref{defn:RPathss}, we compute
\[
\begin{aligned}
\mu_{\cG_1,\xi_1}(\wi)&=\sum_{e\in E}\xi_1(e)\sum_{\substack{\overline q\in \Pt^k(\cG)\\ \text{\emph{end}}(e)=\text{\emph{st}}(\overline q),\\\text{\emph{lab}}(\overline q)=\wi}} p (\overline q)=\sum_{s\in S}\xi(s)\sum_{\substack{e\in E\\\text{\emph{st}}(e)=s} }p(e)\sum_{\substack{\overline q\in \Pt^k(\cG)\\ \text{\emph{end}}(e)=\text{\emph{st}}(\overline q),\\\text{\emph{lab}}(\overline q)=\wi}} p (\overline q) \\\\
&=\sum_{j\in \M}\sum_{s\in S}\xi(s)\sum_{\substack{e\in E\\\text{\emph{st}}(e)=s\\\text{\emph{lab}}(e)=j} }p(e)\sum_{\substack{\overline q\in \Pt^k(\cG)\\ \text{\emph{end}}(e)=\text{\emph{st}}(\overline q),\\\text{\emph{lab}}(\overline q)=\wi}} p (\overline q)=\sum_{j\in \M}\sum_{s\in S}\xi(s)\sum_{\substack{\overline r\in \Pt^{k+1}(\cG)\\ \text{\emph{st}}(\overline r)=s,\\\text{\emph{lab}}(\overline r)=(\wi,j)}} p (\overline r)\\&=\sum_{\wj\in \ell^{-1}([\wi])} \mu_{\cG,\xi}(\wj)=\mu_{\cG,\xi} (\ell^{-1}(\,[\wi])),
\end{aligned}
\]
concluding the proof of~\eqref{eq:MeasureRPathLift}. For~\eqref{eq:EquivalenceMeasureUnderErgodicity}, it suffices to recall Lemma~\ref{lem:LyapunovMeasureRLift}, and the fact that, by Lemma~\ref{lemma:Ergodicity}, $\mu_{\cG,\xi_\cG}$ is shift-invariant and ergodic, i.e.
$
\mu_{\cG,\xi_\cG}(\wi)=\mu_{\cG,\xi} (\ell^{-1}(\,[\wi]))$, $\forall\,\wi\in \M^k,\forall k\in \N$.
For~\eqref{eq:EquivalencesPathsLiftRadius}, the equality and the first  inequality are straightforward by~\eqref{eq:MeasureRPathLift} and~\eqref{eq:EquivalenceMeasureUnderErgodicity}. Let us thus prove, given any family of function $\cV\subset \cF_n$, that  $\rho_{0,\cV}(\cA,\cG_1)\leq \rho_{0,\cV}(\cA,\cG)$. Consider any $\rho>\rho_{0,\cV}(\cA,\cG)$, then recalling there exists $\cW=\{g_s\;\vert\;s\in S\}\subset\cV$, a set of Lyapunov function for $\cS(\cA,\cG)$ w.r.t. $\rho$, i.e., we recall, implies that
\[
\prod_{i\in \M}\prod_{b\in S}f_b(A_ix)^{p_{a,b,i}}\leq \rho f_a(x),\;\;\;\;\forall x\in \R^n, \forall a\in S,
\]
or, rewriting it in a more convenient form for our purposes,
\[
\prod_{e\in E,\,\text{\emph{st}}(e)=a}f_{\text{\emph{end}}(e)}(A_{\text{\emph{lab}}(e)}x)^{p(e)}\leq \rho f_a(x),\,\forall x\in \R^n, \forall a\in S.
\]
We want to construct, from $\cW$, a set of Lyapunov function for $\cG_1$ w.r.t. $\rho$. Let thus define $\cW_1:=\{g_e\equiv f_{\en(e)}\;\vert\;e\in E\}$, consider any $e=(c,a,i)\in E$, we have 
\[
\begin{aligned}
\prod_{\substack{\wt e\in E\\ \st(\wt e)=a}}g_{\wt e}(A_{\lab(\wt e)}x)^{p(\wt e)}&=\prod_{\substack{\wt e\in E\\ \st(\wt e)=a}}f_{\en(\wt e)}(A_{\lab(\wt e)}x)^{p(\wt e)}\leq \rho f_a(x)=\rho g_{e}(x),\;\;\; \forall x\in \R^n,
\end{aligned}
\]
proving that $\cW_1$ is a set of Lyapunov function for $\cS(\cA,\cG_1)$ w.r.t. $\rho$, thus concluding the proof.\hfill $\square$
\end{proof}
\section{Numerical Example}\label{sec:NumExample}
In this section, we consider a positive stochastic switched system as in~\eqref{eq:SwitchedSysIntro}, denoted by $\cS(\cH,\cA)$, with the stochastic graph $\cH$ on $\langle 2\rangle$ in Figure~\ref{Fig:GraphH}, already studied in Examples~\ref{ex:GraphAndKlift}~and~\ref{ex:ExamplePathDepLift}, while the set of positive matrices $\cA=\{A_1,A_2\}\subset\R^{2\times 2}_{\geq 0}$ is defined by
\[
A_1:=\begin{bmatrix}0.5 & 1 \\ 0 & 0.5 \end{bmatrix},\;\;\; A_2:=\begin{bmatrix}1 & \,0 \\ 0.1 & \,1 \end{bmatrix}.
\]
The same set of matrices was studied in~\cite[Example 3.2]{Fang97} (with a different underlying irreducible Markov chain).
We note that system $\cS(\cH,\cA)$ is \emph{first moment unstable (or unstable in mean)} i.e., there exists $x_0\in \R^n$ for which
$
\lim_{k\to \infty}\E_{\cH,\xi_{\cH}}\left(|x(k,x_0,\sigma)|\right)>0$,
i.e., the expected norm with respect to the measure $\mu_{\cH,\xi_\cH}$ does not asymptotically converge to $0$.
\mat{Indeed, by positivity of $A_1$ and $A_2$ and applying~\cite[Theorem 2.4]{JunPro11}, $\cS(\cH,\cA)$ is unstable in mean since the ``averaged'' matrix
\[
B=\xi_\cH(a)A_1+\xi_\cH(b)A_2=\tfrac{3}{11}A_1+\tfrac{8}{11}A_2=\left [\begin{smallmatrix} 19/22 & \;3/11 \\ 4/55 & \;19/22 \end{smallmatrix}\right ]
\]
is Schur unstable, i.e. $\rho(B)>1$.} \mat{It can be also shown that the sufficient and necessary LMI conditions for \emph{mean square stability} illustrated in~\cite{CosFra93} are infeasible for $\cS(\cH,\cA)$.} In what follows we show that $\cS(\cH,\cA)$ is almost sure stable, using the ideas developed in Section~\ref{sec:LIFt}.
\begin{table}[b!]\label{table1}
\vspace{0.3cm}
\centering
\caption{{\footnotesize Numerical upper bounds of the probabilistic spectral radius  obtained with different lifts and different candidate functions templates.}}
\begin{tabular}{ |c|c|c|c|} 
 \hline
 Lift $\cG$: & $\cH$ & $\cH^2$ & $\cH_1$  \\ 
 \hhline{|=|=|=|=|}
  $\rho_{0,\cQ}(\cA,\cG)$ & $1.002$ & $0.896$ & $0.998$ \\
 \hline 
 $\rho_{0,\cD}(\cA,\cG)$ & $1.169$ & $1.045$ & $1.036$ \\
 \hline  
\end{tabular}
\end{table}

\mat{
First, we consider, as set of candidate Lyapunov functions, the set $\cQ\subset \cF_2$ of \emph{quadratic norms}, defined by
$
\cQ:=\{f_Q:\R^2_{\geq 0}\to \R\;\vert\;Q\succ 0\},
$
with $f_Q(x):=\sqrt{x^\top Q x}$. 
To compute an upper bound of the minimal $\rho$ satisfying~\eqref{eq:DefinitionLyapunovFunctionals} we use the techniques presented in~\cite[Corollary 2.2.]{Fang97}. In this case, simply considering the original Markov chain $\cH$, we are unable to provide a stability certificate, since we have obtained $\rho_{0,\cQ}(\cA,\cH)\leq 1.002$.
Instead, considering the lifts $\cH^2$ and $\cH_1$ depicted in Figures~1, we obtain certificates of almost sure stability, since we have that $\rho_{0,\cQ}(\cA,\cH^2)<1$ and $\rho_{0,\cQ}(\cA,\cH_1)<1$. The obtained numeric values are reported in Table I.

For the sake of completeness, we report that we have also considered the set of candidate Lyapunov functions $\cD\subset \cF_2$ of \emph{dual copositive norms} defined by
$
\cD:=\{f^\star_v:\R^2_{\geq0}\to \R\;\vert\;v\in \R^2_{>0}\}$,
where, given $v\in \R^2_{>0}$, we define
$
f^\star_v(x):=\max_{i\in \{1,2\}}\left\{ \frac{x_i}{v_i}  \right\}$, $\forall \;x\in \R^2_{\geq 2}$. These functions are valid norms in this case, since $A_1$ and $A_2$ are positive, for further discussion regarding copositive norms, we refer to~\cite[Section 4]{DebDel21}. This family of functions was chosen for optimization purposes: indeed, inequalities of the form~\eqref{eq:DefinitionLyapunovFunctionals} can be rewritten as follows
\[
\begin{aligned}
\prod_{i\in \M}\prod_{b\in S}(f^\star_{v_b}(A_ix))^{p_{a,b,i}}&\leq\rho f^\star_{v_a}(x),\;\;\forall x\in \R^n_{\geq 0}, \;\;\text{iff}\\
\prod_{i\in \M}\prod_{b\in S}(f^\star_{v_b}(A_iv_a))^{p_{a,b,i}}&\leq\rho ,
\end{aligned}
\] 
simplifying the minimization of the parameter $\rho$ (the upper bound for the probabilistic spectral radius), see~\cite[Section 4]{DebDel21} for the details.
In Table~I we collect the values of the obtained upper bounds on $\rho_0(\cA,\cH,\xi_\cH)$.  

Concluding, we underline how, as proven in Theorem~\ref{thm:TLIFT} and Theorem~\ref{thm:PropertiesPatg}, the $2$-steps lift $\cH_2$ and the path lift of degree~$1$, $\cH_1$, provide a better over appoximation of $\rho_0(\cA,\cH,\xi_\cH)$ with respect to the original graph-based conditions imposed on $\cH$, for these $2$ different sets of candidate functions, $\cQ$ and~$\cD$.}
\section{Conclusion}\label{sec:COncl}
In this work, we have generalized stochastic stability notions, and related algorithms, from arbitrarily switching systems to systems whose switching signal is ruled by a Markov Chain (aka MJLS). Inspired by techniques developed for deterministic switched systems, we presented some numerical schemes to provide tight upper bounds for the probabilistic spectral radius. These techniques rely on formal expansions of the underlying stochastic graph, called \emph{lifts}. Future research will investigate the numerical
aspects related with our approximation technique, and the application of the proposed scheme for more general stochastic systems settings.

\bibliography{biblio} 

\begin{thebibliography}{10}
\providecommand{\url}[1]{#1}
\csname url@samestyle\endcsname
\providecommand{\newblock}{\relax}
\providecommand{\bibinfo}[2]{#2}
\providecommand{\BIBentrySTDinterwordspacing}{\spaceskip=0pt\relax}
\providecommand{\BIBentryALTinterwordstretchfactor}{4}
\providecommand{\BIBentryALTinterwordspacing}{\spaceskip=\fontdimen2\font plus
\BIBentryALTinterwordstretchfactor\fontdimen3\font minus
  \fontdimen4\font\relax}
\providecommand{\BIBforeignlanguage}[2]{{%
\expandafter\ifx\csname l@#1\endcsname\relax
\typeout{** WARNING: IEEEtran.bst: No hyphenation pattern has been}%
\typeout{** loaded for the language `#1'. Using the pattern for}%
\typeout{** the default language instead.}%
\else
\language=\csname l@#1\endcsname
\fi
#2}}
\providecommand{\BIBdecl}{\relax}
\BIBdecl

\bibitem{liberzon}
D.~Liberzon, \emph{Switching in systems and control}.\hskip 1em plus 0.5em
  minus 0.4em\relax Birkha\"user, 2003.

\bibitem{Ber60}
A.~Bergen, ``Stability of systems with randomly time-varying parameters,''
  \emph{IRE Transactions on Automatic Control}, vol.~5, no.~4, pp. 265--269,
  1960.

\bibitem{KatsKras60}
I.~I. Kats and N.~N. Krasovskii, ``On the stability of systems with random
  parameters,'' \emph{Journal of Applied Mathematics and Mechanics}, vol.~24,
  no.~5, pp. 1225--1246, 1960.

\bibitem{FurKes60}
H.~Furstenberg and H.~Kesten, ``{Products of Random Matrices},'' \emph{The
  Annals of Mathematical Statistics}, vol.~31, no.~2, pp. 457 -- 469, 1960.

\bibitem{FangLop1995}
Y.~Fang, K.~A. Loparo, and X.~Feng, ``Stability of discrete time jump linear
  systems,'' \emph{Journal of Mathematical Systems, Estimation and Control},
  vol.~5, no.~3, pp. 275--321, 1995.

\bibitem{CosFra93}
O.~L.~V. Costa and M.~D. Fragoso, ``Stability results for discrete-time linear
  systems with {Markovian} jumping parameters,'' \emph{Journal of Mathematical
  Analysis and Applications}, vol. 179, no.~1, pp. 154--178, 1993.

\bibitem{Pro11}
V.~Y. Protasov, ``Invariant functions for the {Lyapunov} exponents of random
  matrices,'' \emph{Sbornik: Mathematics}, vol. 202, no.~1, pp. 101--126, 2011.

\bibitem{BolCol04}
P.~Bolzern, P.~Colaneri, and G.~De~Nicolao, ``On almost sure stability of
  discrete-time {Markov} jump linear systems,'' in \emph{43rd IEEE Conference
  on Decision and Control (CDC)}, vol.~3, 2004, pp. 3204--3208.

\bibitem{DaiHua08}
X.~Dai, Y.~Huang, and M.~Xiao, ``Almost sure stability of discrete-time
  switched linear systems: A topological point of view,'' \emph{SIAM Journal on
  Control and Optimization}, vol.~47, no.~4, pp. 2137--2156, 2008.

\bibitem{PhiEss16}
M.~Philippe, R.~Essick, G.~E. Dullerud, and R.~M. Jungers, ``Stability of
  discrete-time switching systems with constrained switching sequences,''
  \emph{Automatica}, vol.~72, pp. 242--250, 2016.

\bibitem{LeeDull06}
J.-W. Lee and G.~E. Dullerud, ``Uniform stabilization of discrete-time switched
  and {Markovian} jump linear systems,'' \emph{Automatica}, vol.~42, no.~2, pp.
  205--218, 2006.

\bibitem{Jung09}
R.~M. Jungers, \emph{The Joint Spectral Radius: Theory and Applications}, ser.
  Lecture Notes in Control and Information Sciences.\hskip 1em plus 0.5em minus
  0.4em\relax Springer-Verlag, 2009, vol. 385.

\bibitem{COsFraMar05}
O.~L.~V. Costa, M.~D. Fragoso, and R.~P. Marques, \emph{Discrete-Time Markov
  Jump Linear Systems}, ser. Probability and Its Applications.\hskip 1em plus
  0.5em minus 0.4em\relax Springer-Verlag, 2005.

\bibitem{JunPro11}
R.~M. Jungers and V.~Y. Protasov, ``Fast methods for computing the $p$-radius
  of matrices,'' \emph{SIAM Journal on Scientific Computing}, vol.~33, no.~3,
  pp. 1246--1266, 2011.

\bibitem{FengLopJi92}
X.~Feng, K.~A. Loparo, Y.~Ji, and H.~J. Chizeck, ``Stochastic stability
  properties of jump linear systems,'' \emph{IEEE Transactions on Automatic
  Control}, vol.~37, no.~1, pp. 38--53, 1992.

\bibitem{Fang97}
Y.~Fang, ``A new general sufficient condition for almost sure stability of jump
  linear systems,'' \emph{IEEE Transactions on Automatic Control}, vol.~42,
  no.~3, pp. 378--382, 1997.

\bibitem{ProJun13}
V.~Y. Protasov and R.~M. Jungers, ``Lower and upper bounds for the largest
  {Lyapunov} exponent of matrices,'' \emph{Linear Algebra and its
  Applications}, vol. 438, no.~11, pp. 4448--4468, 2013.

\bibitem{ChiMaz21}
Y.~Chitour, G.~Mazanti, and M.~Sigalotti, ``On the gap between deterministic
  and probabilistic joint spectral radii for discrete-time linear systems,''
  \emph{Linear Algebra and its Applications}, vol. 613, p. 24–45, 2021.

\bibitem{TsiBlo97}
J.~N. Tsitsiklis and V.~D. Blondel, ``The {Lyapunov} exponent and joint
  spectral radius of pairs of matrices are hard—when not impossible—to
  compute and to approximate,'' \emph{Mathematics of Control, Signals and
  Systems}, vol.~10, pp. 31--40, 1997.

\bibitem{LinMar95}
D.~Lind and B.~Marcus, \emph{An Introduction to Symbolic Dynamics and
  Coding}.\hskip 1em plus 0.5em minus 0.4em\relax USA: Cambridge University
  Press, 1995.

\bibitem{Pro10}
V.~Y. Protasov, ``Invariant functionals for random matrices,'' \emph{Functional
  Analysis and Its Applications}, vol.~44, no.~3, pp. 230--233, 2010.

\bibitem{Walters2000}
P.~Walters, \emph{An Introduction to Ergodic Theory}, ser. Graduate Texts in
  Mathematics.\hskip 1em plus 0.5em minus 0.4em\relax Springer New York, 2000.

\bibitem{Pro08}
V.~Y. Protasov, ``Extremal {$L_p$}-norms of linear operators and self-similar
  functions,'' \emph{Linear Algebra and its Applications}, vol. 428, no.~10,
  pp. 2339--2356, 2008.

\bibitem{SutFawRen21}
D.~Sutter, O.~Fawzi, and R.~Renner, ``Bounds on {Lyapunov} exponents via
  entropy accumulation,'' \emph{IEEE Transactions on Information Theory},
  vol.~67, no.~1, pp. 10--24, 2021.

\bibitem{DebDel21}
V.~Debauche, M.~{Della Rossa}, and R.~M. Jungers, ``Comparison of path-complete
  {Lyapunov} functions via template-dependent lifts,'' \emph{Nonlinear
  Analysis: Hybrid Systems}, vol.~46, p. 101237, 2022.

\end{thebibliography}
\bibliographystyle{plain}

\end{document}